\documentclass{article}
\usepackage{indentfirst, latexsym, bm, amssymb, amsmath}

\begin{document}
\newtheorem{theorem}{Theorem}[section]
\newtheorem{corollary}[theorem]{Corollary}
\newtheorem{definition}[theorem]{Definition}
\newtheorem{conjecture}[theorem]{Conjecture}
\newtheorem{question}[theorem]{Question}
\newtheorem{lemma}[theorem]{Lemma}
\newtheorem{proposition}[theorem]{Proposition}
\newtheorem{example}[theorem]{Example}
\newenvironment{proof}{\noindent {\bf
Proof.}}{\rule{3mm}{3mm}\par\medskip}
\newcommand{\remark}{\medskip\par\noindent {\bf Remark.~~}}
\newcommand{\pp}{{\it p.}}
\newcommand{\de}{\em}

\newcommand{\JEC}{{\it Europ. J. Combinatorics},  }
\newcommand{\JCTB}{{\it J. Combin. Theory Ser. B.}, }
\newcommand{\JCT}{{\it J. Combin. Theory}, }
\newcommand{\JGT}{{\it J. Graph Theory}, }
\newcommand{\ComHung}{{\it Combinatorica}, }
\newcommand{\DM}{{\it Discrete Math.}, }
\newcommand{\ARS}{{\it Ars Combin.}, }
\newcommand{\SIAMDM}{{\it SIAM J. Discrete Math.}, }
\newcommand{\SIAMADM}{{\it SIAM J. Algebraic Discrete Methods}, }
\newcommand{\SIAMC}{{\it SIAM J. Comput.}, }
\newcommand{\ConAMS}{{\it Contemp. Math. AMS}, }
\newcommand{\TransAMS}{{\it Trans. Amer. Math. Soc.}, }
\newcommand{\AnDM}{{\it Ann. Discrete Math.}, }
\newcommand{\NBS}{{\it J. Res. Nat. Bur. Standards} {\rm B}, }
\newcommand{\ConNum}{{\it Congr. Numer.}, }
\newcommand{\CJM}{{\it Canad. J. Math.}, }
\newcommand{\JLMS}{{\it J. London Math. Soc.}, }
\newcommand{\PLMS}{{\it Proc. London Math. Soc.}, }
\newcommand{\PAMS}{{\it Proc. Amer. Math. Soc.}, }
\newcommand{\JCMCC}{{\it J. Combin. Math. Combin. Comput.}, }
\newcommand{\GC}{{\it Graphs Combin.}, }

\title{ A Variation of the Erd\H{o}s-S\'{o}s Conjecture in Bipartite Graphs
\thanks{This work is supported by the Joint NSFC-ISF Research Program (jointly funded by the National Natural Science Foundation of China and the Israel Science Foundation (No. 11561141001)), the National Natural Science Foundation of China (Nos.11531001 and 11271256),   Innovation Program of Shanghai Municipal Education Commission (No. 14ZZ016) and Specialized Research Fund for the Doctoral Program of Higher Education (No.20130073110075).
\newline \indent $^{\dagger}$Corresponding author:
Xiao-Dong Zhang (Email: xiaodong@sjtu.edu.cn)}}
\author{ {\bf Long-Tu Yuan $\cdot$ Xiao-Dong Zhang$^{\dagger}$}   \\
{\small School of Mathematical Sciences, MOE-LSC, SHL-MAC
}\\
{\small Shanghai Jiao Tong University} \\
{\small  800 Dongchuan Road, Shanghai, 200240, P.R. China}\\
{\small Email: yuanlongtu@sjtu.edu.cn, xiaodong@sjtu.edu.cn }}
\date{}
\maketitle
\begin{abstract}
  The Erd\H{o}s-S\'{o}s Conjecture states that every graph with average degree more than $k-2$ contains  all  trees of  order $k$ as subgraphs.  In this paper,  we consider a variation of the above conjecture: studying the maximum size of an  $(n,m)$-bipartite graph  which does not contain all  $(k,l)$-bipartite trees for given integers $n\ge m$ and $k\ge l$.  In particular, we determine that the maximum size of an $(n,m)$-bipartite graph  which does not  contain all $(n,m)$-bipartite trees as subgraphs  (or all $(k,2)$-bipartite trees as subgraphs, respectively). Furthermore, all these extremal graphs are characterized.
\end{abstract}


\vskip 0.5cm

\section{Introduction}
The graphs considered in this paper are finite, undirected, and simple (no loops or multiple edges). Let $G=G[V; E]$ be a graph with vertex set $V$ and edge set $E$. The number of  vertices in $V$ is called {\it order} of $G$ and  the number of edges in $E$ is called {\it size} of $G$, denoted by   $e(G)$.
The {\it degree} of $v\in V$, the number of edges incident to $v$, is denoted by $d_{G}(v)$ and the set of neighbors of $v$ is denoted by $N_{G}(v)$.  Moreover, a vertex of degree one  is called a {\it pendent} vertex. If $u$ and $v$ in $V$ are adjacent, we say that $u$ {\it hits} $v$ and $v$ {\it hits} $u$. If $u$ and $v$ are not adjacent, we say that $u$ {\it misses} $v$ and $v$ {\it misses} $u$. The path with $n$ vertices is denoted by $P_{n}$, the star with $n$ vertices is denoted by $K_{1,n-1}$ ($K_{1,0}$ is an isolated vertex, and $K_{1,1}$ is an edge), the cycle with $n$ vertices is denoted by $C_{n}$, and the double star with $k_1+k_2$ vertices which is obtained from two stars $K_{1, k_1-1}$ and $K_{1, k_2-1}$ by joining an edge between two central vertices with degree $k_1-1,k_2-1$ is denoted by $S_{k_1,k_2}$.
  Let $G$ and $H$ be two vertex disjoint graphs. Denote by $G\cup H$ the vertex disjoint union of $G$ and $H$ and by $k\cdot G$ the vertex disjoint union of $k$ copies of a graph $G$.
  In addition, $\delta(G)$, $\Delta(G)$ and $avedeg(G)=\frac{2e(H)}{|V(H)|}$ are denoted by the minimum, maximum and average degree in $V(G)$, respectively. If $S\subseteq V(G)$, the induced subgraph of $G$ by $S$ is denoted by $G[S]$. Let $T$ be a tree of order $k$. If there exists an injection $f:V(T)\rightarrow V(G)$ such that $f(u)f(v)\in E(G)$ if $uv \in E(T)$ for $u, v\in V(T)$, we call $f$ an {\it embedding} of $T$ into $G$ and $G$ contains a copy of $T$ as a subgraph, denoted by $T\subseteq G$.
An {\it $(n,m)$-bipartite graph (or bigraph)} $G[U,V; E]$, or $B_{n,m}$, is a bipartite graph of order $m+n$ whose vertices can be divided into two disjoint sets $U$ and $V$  with $|U|=n$ and $|V|=m$ such that every edge joins one vertex in $U$ to another vertex in $V$.  Moreover, denote by $K_{n,m}$ the complete  $(n,m)$-bipartite graph.
  Furthermore, denote by  $\mathbb{ B}_{n,m}$ the set of all bipartite graphs $B_{n,m}$, and  $\mathbb{T}_{n,m}$ the set of all the $(n,m)$-bipartite trees $T_{n,m}$ with two partitions $|U|=n$ and $|V|=m$, respectively.

  We call a problem {\it a $Tur\acute{a}n$ type extremal problem} if  a family $\mathbb{L}$ of graphs  are given from  universe,  such as $G_{n}$ is a graph of order $n$, we try to maximize $e(G_n)$ of $G_n$  under the condition that $G_{n}$  does not contain  $L\in \mathbb{L}$. The maximum value is denoted by $ex(n,\mathbb{L})$. Similarly, for a given family $\mathbb{L}$ of bipartite graphs, the maximum value $e(B_{n,m})$ of $B_{n,m}$ under the condition that $B_{n,m}$ does not contain $L\in \mathbb{L}$ is denoted by $ex(n,m;\mathbb{L})$. Furthermore, if a bipartite graph $B_{n,m}$ with $ex(n,m;\mathbb{L})$ edges does not contain  $ L\in \mathbb{L}$, then this bipartite graph is called {\it an extremal bipartite graph} for $ \mathbb{L}$.
In 1959, Erd\H{o}s and Gallai \cite{erdHos1959maximal} proved the following theorem.
\begin{theorem}\label{erdod1959} Let $G$ be a graph with $avedeg(G)>k-2$. Then $G$ contains a path of order $k$ as a subgraph.
\end{theorem}
Based on the above theorem and related results,  Erd\H{o}s and S\'{o}s   proposed the following  well known conjecture (for example, see \cite{erdos1965}).
\begin{conjecture}\label{con} Let $G$ be a graph with $avedeg(G)>k-2$.
Then $G$ contains all trees of order $k$.
Furthermore,
$$ex(n, \mathbb{T}_k)=\lfloor\frac{(k-2)n}{2}\rfloor,$$
where $\mathbb{T}_k$ is the set of all trees of order $k$.
\end{conjecture}

In \cite{Ajtai1, Ajtai2,Ajtai3}, Ajtai, Koml\'{o}s, Simonovits and  Szemer\'{e}di proved that the Erd\H{o}s-S\'{o}s Conjecture is true for sufficiently large $k$. Fan \cite{fan2013erdHos} proved that the Erd\H{o}s-S\'{o}s Conjecture holds for the spiders of large size. More results on this conjecture can be referred to \cite{balasubramanian2007erdos,brandt1996erdHos,dobson2002constructing,eaton2010erdos,mclennan2005erdHos,sacle1997erdHos,sidorenko1989asymptotic,tiner2010erdos,wozniak1996erdos,Yuan2016erdHos,Zhou1984erdHos}. On the extremal problems on complete bipartite graph, K\H{o}v\'{a}ri, S\'{o}s and Tur\'{a}n \cite{kovari1954} proved the following result:
   \begin{theorem}\cite{kovari1954}\label{kovari}
   The maximum size of a graph containing no complete bipartite graph $K_{a,b}$ is at most
   $\frac{1}{2}\sqrt[a]{b-1}n^{2-1/a}+\frac{1}{2}(a-1)n$.
   \end{theorem}

  F\"{u}redi and Simonovits \cite{Furedi2013} written a survey on extremal graph theory focusing on the cases when one of the excluded graphs is bipartite.
For example,  Gy\H{o}ri \cite{gyori1995}  proved that $ex(n,m;C_{6})<2n+\frac{m^2}{2}$.
Gy\'{a}rf\'{a}s, Rousseau and Schelp \cite{gyarfas1984} proved that the following theorem.
\begin{theorem} \cite{gyarfas1984}
\label{P2l}
Let $n\geq m$. Then \\
$$ ex(n,m;P_{2l})=\left\{\begin{array}{ll}
 nm, &\mbox {for} \ \ m\leq l-1;\\
 (l-1)n, & \mbox{for} \ \ l-1 <m<2(l-1);\\
 (l-1)(n+m-2l+2), & \mbox{for}\ \  m\geq 2(l-1).\end{array}\right.
 $$
 Furthermore, (1). If $m\leq l-1$, then all extremal graphs are $K_{n,m}$.\\
 (2). If $l-1<m<2(l-1)$, then all extremal graphs are $K_{l-1,n}\cup (m-l+1)\cdot K_1$.\\
 (3). If $m\geq 2(l-1)$, then all extremal graphs are $K_{l-1,m-l+1}\cup K_{l-1,n-l+1}$; or $K_{l-1,i}\cup K_{l-1,n-i}$ for $i=0,1,\ldots,\lfloor\frac{n}{2}\rfloor$, when $m=2(l-1)$.
\end{theorem}
Moreover, they \cite{gyarfas1984}  also determined $ex(n,m;P_{2l+1})$ and the extremal graphs.  The related results about the extremal graphs with focusing on the case when one of the excluded graphs is bipartite can be referred to \cite{Balbuena2007,Furedi2006,Sarkozy1995}. Motivated by Erd\H{o}s-S\'{o}s Conjecture and the above results, in this paper, we propose the following problem.
\begin{question}
 Determine $ex(n,m;\mathbb{T}_{k,l})$ and characterize all extremal graphs, where  $\mathbb{T}_{k,l}$  is set of all $(k,l)$-bipartite trees of order $k+l$.
 \end{question}

 The main results in this paper are stated as follows.

\begin{theorem}
\label{main1}
  Let $n\geq m$. Then $$ex(n,m;\mathbb{T}_{n,m})= (n-1)m.$$
 Furthermore,  (1). If $n=m$, then all extremal graphs for $\mathbb{T}_{n,n}$ are $(n,n)$-bipartite graphs $G[U,V; E]$ such that the degree of each vertex in $U$ (or $V$) is $n-1$.

      (2). If $n=m+1$, then all extremal graphs for $\mathbb{T}_{n,m}$ are $(n,m)$-bipartite graphs $G[U,V; E]$ such that the degree of each vertex in  $V$ is $n-1$, or  $      d(v_{1})=1$ and $d(v_{2})=\ldots =d(v_{m})=n$.

       (3). If $n> m+1$,  then all extremal graphs for $\mathbb{T}_{n,m}$ are $(n,m)$-bipartite graphs $G[U,V; E]$ such that the degree of each vertex in $V$ is $n-1$.
   \end{theorem}
   Moreover, for small $k$ and $l$ we have the following results.
 \begin{theorem}
\label{main2}
Let $n\geq m\geq2$ and $n\geq k$.

  (1). If $k=2$, then   $ex(n,m;\mathbb{T}_{k,2})=n+m-2$.

  (2). If $k\geq3$ and  $m=2$, then
  $$ex(n,2;\mathbb{T}_{k,2})=\left\{\begin{array}{ll}
   n+\lceil\frac{k}{2}\rceil-1, & \mbox{for}\ \ n\geq\lfloor\frac{3k}{2}\rfloor-1;\\
  2(k-1),& \mbox{for} \ \  n\leq\lfloor\frac{3k}{2}\rfloor-1.\end{array}\right.$$

(3). If  $ 3\leq m\leq k$, then
$$ ex(n,m;\mathbb{T}_{k,2})=\left\{\begin{array}{ll}
(m-2)(k-1)+n, & \mbox{for }\ \  n\geq 2k-1;\\
  m(k-1), & \mbox{for } \ \ n\leq2k-2.\end{array}\right.$$

(4). If $ m\geq k+1\ge 4$, then
  $$ex(n,m;\mathbb{T}_{k,2})=\left\{\begin{array}{ll}
  (k-1)(m-1)+n-m+1, & \mbox{for}\ \  n-m\geq k-1;\\
 (k-1)m, & \mbox{for } \ \ n-m\leq k-2.\end{array}\right.$$
 \end{theorem}
\begin{theorem}
\label{main3}
  Let $n\geq m\geq 3$.
Then
  $$ ex(n,m;\mathbb{T}_{3,3})=\left\{\begin{array}{ll} 2(n+m)-8,  & \mbox{for} \ \
   n\geq5,m\geq 5;\\
  9, & \mbox{for} \ \ n=m=4;\\
     2n, & \mbox{for\ else}.\end{array}\right.$$
 \end{theorem}

 The rest of this paper is organized as follows. In Sections 2, 3 and 4, the proofs of Theorems~\ref{main1}, \ref{main2} and \ref{main3} are presented, respectively. Furthermore, all extremal graphs in Theorem~\ref{main1}, \ref{main2} and \ref{main3} are characterized.

\section{Proof of Theorem~\ref{main1}}
 First we prove a simple result: $ex(B_{n,n},C_{2n})=n(n-1)+1$, and characterize all the extremal graphs.
The proof of this result depends on the following  result \cite{chvatal1972} by Chv\'{a}tal which strengthens a result \cite{MM1962} of Moon and Moser on Hamiltonian cycles in bipartite graphs.
 \begin{lemma}\cite{MM1962}\label{1}
  Let $B_{n,n}=G[U,V;E] $ be a bipartite graph with $U=\{u_{1},\ldots,u_{n}\}$ and $V=\{v_{1},\ldots,v_{n}\}$, where $d(u_{1})\leq \ldots \leq d(u_{n})$ and $d(v_{1})\leq \ldots \leq d(v_{n})$. If
 \begin{align}
  d(u_{k})\leq k < n \mbox{ implies } d(v_{n-k})\geq n-k+1,  \label{eq1}
 \end{align}
   then $B_{n,n}$ is Hamiltonian. If \eqref{eq1} does not hold, then there exists a non-Hamiltonian bipartite graph $H=G^{\prime}[X,Y; E]$ with $X=\{x_{1},\ldots,x_{n}\}$ and $Y=\{y_{1},\ldots,y_{n}\}$, where $d(x_{1})\leq \ldots \leq d(x_{n})$ and $d(y_{1})\leq \ldots \leq d(y_{n})$, such that $d(u_{k})\leq d(x_{k})$ and $d(v_{k})\leq d(y_{k})$, for $k=1,\ldots,n$.
\end{lemma}
\begin{lemma}\label{2}
  Let $B_{n,n}=G[U,V;E]$ be a bipartite graph with $|U|=|V|=n\geq 2$.
  Then
   $$ex(n,n; C_{2n})=n^2-n+1.$$
   Furthermore, if a bipartite graph $B_{n,n}$ with $n^2-n+1 $ edges does not contain $C_{2n}$ as a subgraph, then $B_{n,n}=K_{n, n-1}+e.$ In other words, if
  \begin{align}
        e(B_{n,n})\geq n^2-n+1  \label{eq2}
  \end{align}
 holds, then $G$ is Hamiltonian unless $B_{n,n}$ is the graph obtained from $K_{n,n-1}$ by adding a pendent edge, i.e.,  $B_{n,n}=K_{n,n-1}+e$.
  \end{lemma}
 \begin{proof} Let $U=\{u_1, \ldots, u_n\}$ and $V=\{v_1, \ldots, v_n\}$ with $d(u_{1})\leq \ldots \leq d(u_{n})$ and $d(v_{1})\leq \ldots \leq d(v_{n})$.  Suppose that $B_{n,n}$ is not Hamiltonian.  Then by Lemma~\ref{1}, there is a vertex $u_{k}$ such that $d(u_{k})\leq k<n$ and $d(v_{n-k})\leq n-k.$
   Hence  $n^2-n+1\le e(B_{n,n})\leq k^2+(n-k)n$, which implies $k=1$ or $k=n-1$.
   If $k=n-1,$ then $d(v_{1})\leq 1$, so $d(v_{1})=1,d(v_{2})=\ldots =d(u_{n})=n,d(u_{1})=\ldots =d(u_{n-1})=n-1,$ and $d(u_{n})=n,$ and hence $B_{n,n}= K_{n,n-1}+e$. If $k=1$, then $d(u_{1})=1,d(u_{2})=\ldots =d(u_{n})=n,d(v_{1})=\ldots =d(v_{n-1})=n-1,$ $d(v_{n})=n,$  and hence $B_{n,n}= K_{n,n-1}+e$.
\end{proof}
In order to prove Theorem~\ref{main1}, we need some lemmas.
  \begin{lemma}\label{pendent vertex n,n}
Let $T_{n,m}=T[U,V;E]$ be a  tree with $|U|=n$ and $|V|=m$. If all the pendent vertices of $T$ are in $U$, then $n>m$.
\end{lemma}
\begin{proof}  We prove the statement by induction on $n+m$. Since all the pendent vertices of $T$ are in $U$, $n+m\ge 3$. It is easy to see the assertion holds for $n+m=3$.
   Assume the assertion holds for $n+m-1$.  Now let $u$ in $U$  be a pendent vertex of $T_{n,m}$ and $v$  be its neighbour in $V$. If $v$ is not a pendent vertex of $T_{n-1,m}=T_{n,m}-\{u\}$, then all pendent vertices in $T_{n-1, m}$   are  in $U\setminus \{u\}$ and by the induction hypothesis,  $n-1>m$. So the assertion holds. If  $v$ is  a pendent vertex of $T_{n-1,m}=T_{n,m}-\{u\}$,  then all pendent vertices in  $T_{n-1,m-1}=T_{n,m}-\{u,v\}$ are in $U\setminus\{u\}$. Hence by the induction hypothesis, $n-1>m-1$, which implies $|U|>|V|$.
\end{proof}
\begin{lemma}\label{pendent vertex n,m}
Let $T_{n,m}=T[U,V;E]$ be a tree with  $|U|=n\geq|V|=m$. Then there are at least $n-m+1$ pendent vertices in $U$.
\end{lemma}
\begin{proof} Suppose that there are at most $n-m$ pendent vertices in $U$. Then
let $T_{n^{\prime},m}$ be the tree obtained from $T_{n,m}$ by deleting all pendent vertices in $U$. Then  $ n^{\prime}\geq m$ and all  pendent vertices of $T_{n^{\prime},m}$ are in $V$.  Hence by Lemma~\ref{pendent vertex n,n},  we have $m>n^{\prime}$, which  is  a contradiction. So the assertion holds.
\end{proof}

Moreover, we need the following notation.

  \begin{definition}\label{strongly}
  Let $B_{n,n}=G[U,V;E]$ be a bipartite graph with $|U|=|V|=n$ and  $T_{n,n}=T[U^{\prime},V^{\prime};E^{\prime}]$ be a tree with $|U^{\prime}|=|V^{\prime}|=n$. If there exist two {\it embeddings}, says,  $f: T_{n,n}\rightarrow B_{n,n}$ and $g: T_{n,n}\rightarrow B_{n,n}$,  such that $f(U^{\prime})=U, f(V^{\prime})=V;$ and $g(V^{\prime})=U, g(U^{\prime})=V$, then we say that $B_{n,n}$ {\it strongly contains} $T_{n,n}$ as a subgraph.
  \end{definition}

\begin{lemma}\label{strongly,m,m}
Let $B_{n,n}=G[U,V;E]$ be a bipartite graph with $U=\{u_{1},\ldots,u_{n}\}$ and $V=\{v_{1},\ldots,v_{n}\}$ such that $d(u_{1})\leq \ldots \leq d(u_{n})$ and $d(v_{1})\leq \ldots \leq d(v_{n})$, where $n\ge 3$. If $e(B_{n,n})\geq n(n-1)$, then
$B_{n,n}$ strongly contains all trees in $\mathbb{T}_{n,n}$ as subgraphs, unless the degree of each vertex in $U$ (or $V$) is $n-1$.
 \end{lemma}
\begin{proof} It is sufficient for us to prove that $B_{n,n}$ with $e(B_{n,n})=n(n-1)$ and $d(u_n)=d(v_n)=n$  strongly contains all trees in $\mathbb{T}_{n,n}$ as subgraphs for $n\ge 3$.
 For $n=3$, it follows from Figure 1 that the assertion holds.

\begin{picture}(200,80)(0,-50)
\put(0,20){\line(1,0){20}}
\put(0,20){\line(1,-1){20}}
\put(0,20){\line(1,-2){20}}
\put(0,0){\line(1,1){20}}
\put(0,0){\line(1,0){20}}
\put(0,-20){\line(1,2){20}}

\put(0,20){\circle*{2}}
\put(20,20){\circle*{2}}
\put(0,0){\circle*{2}}
\put(20,0){\circle*{2}}
\put(0,-20){\circle*{2}}
\put(20,-20){\circle*{2}}
\put(10,-30){$G_{0}$}

\put(40,20){\line(1,0){20}}
\put(40,20){\line(1,-1){20}}
\put(40,20){\line(1,-2){20}}
\put(40,0){\line(1,1){20}}
\put(40,-20){\line(1,2){20}}

\put(40,20){\circle*{2}}
\put(60,20){\circle*{2}}
\put(40,0){\circle*{2}}
\put(60,0){\circle*{2}}
\put(40,-20){\circle*{2}}
\put(60,-20){\circle*{2}}
\put(50,-30){$G_{1}$}

\put(80,20){\line(1,0){20}}
\put(80,20){\line(1,-1){20}}
\put(80,20){\line(1,-2){20}}
\put(80,0){\line(1,0){20}}
\put(80,-20){\line(1,0){20}}

\put(80,20){\circle*{2}}
\put(100,20){\circle*{2}}
\put(80,0){\circle*{2}}
\put(100,0){\circle*{2}}
\put(80,-20){\circle*{2}}
\put(100,-20){\circle*{2}}
\put(90,-30){$G_{2}$}

\put(120,20){\line(1,0){20}}
\put(120,0){\line(1,1){20}}
\put(120,0){\line(1,0){20}}
\put(120,-20){\line(1,0){20}}
\put(120,-20){\line(1,1){20}}

\put(120,20){\circle*{2}}
\put(140,20){\circle*{2}}
\put(120,0){\circle*{2}}
\put(140,0){\circle*{2}}
\put(120,-20){\circle*{2}}
\put(140,-20){\circle*{2}}
\put(130,-30){$G_{3}$}

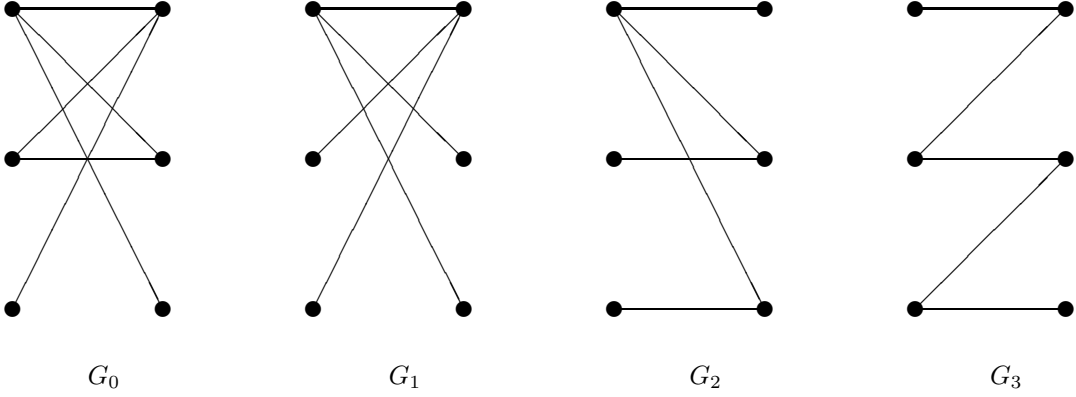
\put(0,-40){Figure 1: $G_{0}$ is the only $(3,3)$-bipartite graph with six edges and $d(u_3)=d(v_3)=3$.}
\put(45,-45){$G_{1},G_{2},G_{3}$ are all trees in $\mathbb{T}_{3,3}$}
\end{picture}

   Assume the assertion holds for $n-1$.  Let $T_{n,n}=G[U^{\prime},V^{\prime};E^{\prime}]$ be any tree in $\mathbb{T}_{n,n}$ with $|U^{\prime}|=|V^{\prime}|=n$.
  We may assume that $e(B_{n,n})=n(n-1)$, $d(u_n)=n$ and $d(v_1)\leq n-2$. Let $B^{\prime}_{n-1,n-1}=B_{n,n}-\{u_n,v_1\}$. Note that
  $$\begin{array}{lll}
  e(B^{\prime}_{n-1,n-1})&=& e(B_{n,n})-(d(u_n)+d(v_1)-1)\\
  &\geq&n(n-1)-(n+n-2-1)>(n-1)(n-2).
  \end{array}$$
  Let $y_0$ be a pendent vertex of $T_{n,n}$ and $x$ be the unique neighbor of $y_0$ in $T_{n,n}$, say, $y_0\in V^{\prime}$ and $x\in U^{\prime}$. Let $\{y_0,y_1,\ldots,y_t\}$ be the set of all neighbors of $x$ in $T_{n,n}$. So, $1\leq t\leq n-1$. Note that $T_{n,n}-\{x,y_0\}$ consists of $t$ components each containing exactly one vertex in $\{y_1,\ldots,y_t\}$. Let $x^{\prime}$ be a vertex in $T_{n,n}$ that belongs to the partite set same as $x$. We may assume that $x^{\prime}$ and $y_1$ are contained in the same component in $T_{n,n}-\{x,y_0\}.$ Then let $T^{\prime}_{n-1,n-1}$ be the graph obtained from $T_{n,n}-\{x,y_0\}$ by adding the edges $x^{\prime}y_i$ for $2\leq i \leq t$. Note that $T^{\prime}_{n-1,n-1}$ is a tree in $\mathbb{T}_{n-1,n-1}$. Then by the  induction hypothesis, there exists an embedding $f^{\prime}:T^{\prime}_{n-1,n-1}\rightarrow B^{\prime}_{n-1,n-1}$ (recall that $e(B^{\prime}_{n-1,n-1})> (n-1)(n-2)$, and hence the exceptional case cannot occur).  Then let $f$ be the mapping from $T_{n,n}$ to $B_{n,n}$ defined as follows:
  \begin{itemize}
  \item For any vertex $z$ with $z\neq x,y_0$, we have $f(z)=f^{\prime}(z)$.
  \item Define $f(x)=u_n$, and $f(y_0)=v_1$.
  \item For any edge $zw$ in $T_{n,n}$, if $\{z,w\}\cap\{x,y_0\}=\emptyset$, then $f(zw)=f^{\prime}(zw)$.
  \item For $1\le i\le t$, define $f(xy_i)$ as the edge in $B_{n,n}$ connecting $u_n$ and $f(y_i)$ (since $d(u_n)=n$, such an edge must exist.).
  \item Define $f(xy_0)$ as the edge in $T_{n,n}$ connecting $u_n$ and $v_1$.
\end{itemize}
We see that this map is indeed an embedding of $T_{n,n}$ into $B_{n,n}$ with $f(U^{\prime})=U$ and $f(V^{\prime})=V.$ Since we may also assume that $d(v_n)=n$ and $d(u_1)\leq n-2$, similarly,  there exists  an embedding $g$ of $T_{n,n}$ into $B_{n,n}$ with $g(V^{\prime})=U$ and $g(U^{\prime})=V.$ So we finish our proof.
\end{proof}

\begin{lemma}\label{main1'}
Let $B_{n,m}=G[U,V;E]$ be a bipartite graph with $U=\{u_{1},\ldots,u_{n}\}$ and $V=\{v_{1},\ldots,v_{m}\}$ such that $d(u_{1})\leq \ldots \leq d(u_{n})$ and $d(v_{1})\leq \ldots \leq d(v_{m})$,  where $n> m\geq 1$.
If  $e(B_{n,m})\geq m(n-1)$ holds,
  then $B_{n,m}$ contains all trees in $\mathbb{T}_{n,m}$ as subgraphs, unless $d(v_{1})=\ldots =d(v_{m})=n-1$,  or $n-m=1$, $d(v_{2})=\ldots =d(v_{m})=n$ and $d(v_{1})=1$.
\end{lemma}

\begin{proof}  Assume that $B_{n,m}$ with $e(B_{n,m})\ge m(n-1)$ does not satisfy that $d(v_{1})=\ldots =d(v_{m})=n-1$,  or $n-m=1$, $ d(v_{2})=\ldots =d(v_{m})=n$ and $d(v_{1})=1$. We will prove this lemma by induction on $m$. For $m=1,2$, it is trivial and assume that $m\geq 3$. Let $d(v_1)\leq \ldots \leq d(v_m)$ and $d(u_1)\leq \ldots \leq d(u_n)$. Since $e(B_{n,m})\ge m(n-1)$, we can assume that $d(v_m)=n$, and $d(v_1)\geq 2$ when $n=m+1$. Let $T_{n,m}=G[U^{\prime},V^{\prime};E^{\prime}]$ with $|U^{\prime}|=n,|V^{\prime}|=m$ be any tree in $\mathbb{T}_{n,m}$. Since $n>m$, by Lemma~\ref{pendent vertex n,m}, $T_{n,m}$ has a pendent vertex in $U^{\prime}$. Let $y_0\in U^{\prime}$ be a pendent vertex of $T_{n,n}$ and $x\in V^{\prime}$ be the unique neighbor of $y_0$ in $T_{n,m}$. Let $T^{\prime}_{n-1,m-1}$ be obtained by the same way as Lemma~\ref{strongly,m,m}. We consider the following four cases.

{\bf Case 1.} $d(u_1)\leq m-2$. Let $B^{\prime}_{n-1,m-1}=B_{n,m}-\{u_1,v_m\}$. Then
  $$\begin{array}{lll}
  e(B^{\prime}_{n-1,m-1})&=& e(B_{n,m})-(d(u_1)+d(v_m)-1)\\
  &\geq&m(n-1)-(n+m-2-1)>(m-1)(n-2).
  \end{array}$$

{\bf Case 2.} $n>m+1,d(v_{m-1})=n$. Let $B^{\prime}_{n-1,m-1}=B_{n,m}-\{u_1,v_m\}$. Then
  $$\begin{array}{lll}
  e(B^{\prime}_{n-1,m-1})&=& e(B_{n,m})-(d(u_1)+d(v_m)-1)\\
  &\geq&m(n-1)-(n+m-1-1)=(m-1)(n-2).
  \end{array}$$
  and $d_{B^{\prime}_{n-1,m-1}}(v_{m-1})=n-1$.

{\bf Case 3.} $n>m+1,d(u_1)=m-1$ and $d(v_{m-1})<n$. Since $e(B_{n,m})\ge m(n-1)$, we have $m-1=d(u_1)\le \ldots \le d(u_n)$. Moreover, we have $ d(v_{m-1})=n-1$ and $v_{m-1}$ misses $u_i$ for some $i\in \{1,2,\ldots,m\}.$ Otherwise $e(B_{n,m})\leq n+(n-2)(m-1)<m(n-1)$, a contradiction. Let $B^{\prime}_{n-1,m-1}=B_{n,m}-\{u_i,v_m\}$. Then
$$\begin{array}{lll}
  e(B^{\prime}_{n-1,m-1})&=& e(B_{n,m})-(d(u_i)+d(v_m)-1)\\
  &\geq&m(n-1)-(n+m-1-1)=(m-1)(n-2).
  \end{array}$$
  and $d_{B^{\prime}_{n-1,m-1}}(v_{m-1})=n-1$.

{\bf Case 4.} $n=m+1, d(u_1)=m-1$ and $ d(v_1)\geq 2$. We consider the following two subcases. {\bf Subcase 4.1.} $e(B_{n,m})>m(n-1)$. Let $B^{\prime}_{n-1,m-1}=B_{n,m}-\{u_1,v_m\}$.  Then
  $$\begin{array}{lll}
  e(B^{\prime}_{n-1,m-1})&=& e(B_{n,m})-(d(u_{1})+d(v_m)-1)\\
  &>&m(n-1)-(n+m-1-1)=(m-1)(n-2).
  \end{array}$$
  {\bf Subcase 4.2.} $e(B_{n,m})=m(n-1)$. Then $d(u_1)=\ldots =d(u_n)=m-1$ and $d(v_1)\le n-2$. Moreover, there must be a vertex, say $v_{j2}\neq v_1$ missing $u_{i1}\in\{u_1,\ldots,u_n\}$. Otherwise, all of $\{u_1\ldots,u_n\}$ miss $v_1$, hence $d(v_1)=0$, a contradiction. By $d(v_1)\le n-2$, there exists a vertex $u_{i2}\neq u_{i1}$  such that   $u_{i2}$ misses $v_1$. Let $B^{\prime}_{n-1,m-1}=B_{n,m}-\{u_{i2},v_m\}.$  Then
$$\begin{array}{lll}
  e(B^{\prime}_{n-1,m-1})&=& e(B_{n,m})-(d(u_{i2})+d(v_m)-1)\\
  &=&m(n-1)-(n+m-1-1)=(m-1)(n-2).
  \end{array}$$
Let $d_{B^{\prime}_{n-1,m-1}}(v_1^{\prime})\leq \ldots \leq d_{B^{\prime}_{n-1,m-1}}(v_{m-1}^{\prime})$. Since $d_{B^{\prime}_{n-1,m-1}}(v_1)<n-1$, $e(B^{\prime}_{n-1,m-1})=(m-1)(n-2)$ and $u_{i1}$ misses $v_{j2}$ in $B^{\prime}_{n-1,m-1}$, we get $d_{B^{\prime}_{n-1,m-1}}(v_1^{\prime})\geq 2$.

It is easy to see that in all of the above cases $B^{\prime}_{n-1,m-1}$ satisfies the condition of Lemma~\ref{main1'} on $m-1$. Since $d(v_m)=n$ and by the induction hypothesis, $B^{\prime}_{n-1,m-1}$ contains all trees in  $\mathbb{T}_{n-1,m-1}$ as subgraphs. Similarly, by the method of Lemma~\ref{strongly,m,m}, we can find an embedding $f:T_{n,m}\rightarrow B_{n,m}$, and we have done.
\end{proof}

       We are ready to prove the main Theorem~\ref{main1}.

       \begin{proof} It follows from Lemmas~\ref{strongly,m,m} and \ref{main1'} that Theorem~\ref{main1} holds.
       \end{proof}
\section{Proof of Theorem~\ref{main2}}
Theorem~\ref{main2} can be proven by  the following several Lemmas.
\begin{lemma}
\label{main2,1}
  Let $n\geq m\geq 2$. Then $ex(n,m;\mathbb{T}_{2,2})=n+m-2.$ Moreover, if $m=2$, then all extremal graphs are $K_{1,p}\cup K_{1, n-p}$ for $p=0, \ldots, \lfloor\frac{n}{2}\rfloor$;  if $m\geq3$, then all extremal graphs are $K_{1,n-1}\cup K_{1,m-1}$.
 \end{lemma}
\begin{proof} Clearly there is only one tree $P_{4}$ in $\mathbb{T}_{2,2}$.
Let $B_{n,m}$ be an $(n,m)$-bipartite graph with $e(B_{n,m})\geq n+m-1$. If $B_{n,m}$ is connected,  then  $B_{n,m}$ contains $P_4$ as a subgraph; if $B_{n,m}$ is disconnected, then there is a component which contains a cycle, hence $B_{n,m}$ contains $P_4$ as a subgraph. Moreover, $K_{1,n-1}\cup K_{1,m-1}$ does not contain $P_4$ as a subgraph. Hence $ex(n,m;\mathbb{T}_{2,2})=n+m-2.$
 Furthermore, let $B_{n,m}$ be any  extremal graph with $n+m-2$ edges, then $B_{n, m}$ has exactly two components which are stars. Hence  $B_{n,m}=K_{1,p}\cup K_{1, n-p}$ for $p=0, \ldots, \lfloor\frac{n}{2}\rfloor$ when $m=2$ and $B_{n,m}=K_{1,m-1}\cup K_{1, n-1}$ when $m\ge 3$.
\end{proof}

The following simple proposition is  useful  for the proof of the next four lemmas.

 \begin{proposition}\label{prop} Let $B_{n,m}$ be any graph which does not contain all trees in $\mathbb{T}_{k,2}$ as subgraphs. Then
  $$N(u)\cap N(v)=\emptyset$$  for any two vertices $u,v$ with $d(u)>\lceil\frac{k}{2}\rceil-1$  and  $d(v)>k-1$ in the same partition set.
\end{proposition}

\begin{lemma}
\label{main2,2}
  Let $n\geq m= 2$. If $k\geq 3$, then
    $$ex(n,2;\mathbb{T}_{k,2})=\left\{\begin{array}{ll}
   n+\lceil\frac{k}{2}\rceil-1, & \mbox{for}\ \ n\geq\lfloor\frac{3k}{2}\rfloor-1;\\
  2(k-1),& \mbox{for} \ \  n\leq\lfloor\frac{3k}{2}\rfloor-1.\end{array}\right.$$

  Furthermore, (1).  If $n>\lfloor\frac{3k}{2}\rfloor-1$, then  all extremal graphs for $\mathbb{T}_{k,2}$ are $(n,2)$-bipartite graphs $B_{n,2}=G[U,V;E]$ with  $d(v_1)=n,$ and $d(v_2)=\lceil\frac{k}{2}\rceil-1$, where $V=\{v_1, v_2\}$.

   (2). If $n<\lfloor\frac{3k}{2}\rfloor-1$, then all extremal graphs for $\mathbb{T}_{k,2}$ are $(n,2)$-bipartite graphs $B_{n,2}=G[U,V;E]$ with $d(v_1)=d(v_2)=k-1$, where $V=\{v_1, v_2\}$.

    (3). If $n=\lfloor\frac{3k}{2}\rfloor-1$, then  all extremal graphs for $\mathbb{T}_{k,2}$ are  $(n,2)$-bipartite graphs $B_{n,2}=G[U,V;E]$ with  $d(v_1)=n$ and $d(v_2)=\lceil\frac{k}{2}\rceil-1$, or  $d(v_1)=d(v_2)=k-1$, where $V=\{v_1, v_2\}$.
      \end{lemma}
 \begin{proof}
  Let $T_{k,2}$ be any $(k,2)$-bipartite tree in $\mathbb{T}_{k,2}$. Without loss of generality, we may assume that  $T_{k,2}=G^{\prime}[U^{\prime}, V^{\prime}; E^{\prime}]$ with $U^{\prime}=\{u_1^{\prime}, \ldots, u_k^{\prime}\}$ and $V^{\prime}=\{v_1^{\prime}, v_2^{\prime}\}$ such that $N(v_1^{\prime})=\{u_1^{\prime}, \ldots, u_r^{\prime}\}$ and
  $N(v_2^{\prime})=\{u_r^{\prime}, \ldots, u_k^{\prime}\}$, where $1\le r\le  \lceil\frac{k}{2}\rceil$. We consider the following three cases.

   {\bf Case 1.} $n>\lfloor\frac{3k}{2}\rfloor-1$. Let $B_{n,2}=G[U,V;E]$ be any $(n,2)$-bipartite graph with $e(B_{n,2})\geq n+\lceil\frac{k}{2}\rceil$, where $V=\{v_1, v_2\}$. Then $l=|N(v_1)\cap N(v_2)|\ge \lceil\frac{k}{2}\rceil$.
  Hence  assume that $N(v_1)=\{u_1, \ldots, u_p, u_{p+1}, \ldots, u_{p+l}\}$ and
$ N(v_2)=\{u_{p+1}, \ldots, u_{p+l}, u_{p+l+1},$ $ \ldots, u_{p+l+q}\}$ with $0\le p\le q.$
Since $$q+l\ge \frac{p+l+q+l}{2}=\frac{e(B_{n,2})}{2}\ge \frac{n+\lceil\frac{k}{2}\rceil}{2}\ge \frac{\lfloor\frac{3k}{2}\rfloor+\lceil\frac{k}{2}\rceil}{2}\ge k,$$ $B_{n,2}$ contains all trees in $\mathbb{T}_{k,2}$ as subgraphs. If $B_{n,2}=G[U, V; E]$ is any $(n,2)$-bipartite graph with  $ d(v_1)=n$ and $d(v_2)=\lceil\frac{k}{2}\rceil-1,$ then $B_{n,2}$ does not contain
all trees in $\mathbb{T}_{k,2}$ as subgraphs, in particular, a tree $T_{k,2}$ with $d(v_1^{\prime})=d(v_2^{\prime})=\lceil\frac{k}{2}\rceil$ when $k$ is odd or $d(v_1^{\prime})=\lceil\frac{k}{2}\rceil+1,d(v_2^{\prime})=\lceil\frac{k}{2}\rceil$ when $k$ is even. Hence  $ex(n,2;\mathbb{T}_{k,2})=   n+\lceil\frac{k}{2}\rceil-1.$  Now assume that an $(n,2)$-bipartite graph $B_{n,2}=G[U, V; E]$  with $n+\lceil\frac{k}{2}\rceil-1 $ edges  does not contain all trees in $\mathbb{T}_{k,2}$ as subgraphs. It is easy to see that $V=\{v_1,v_2\}$ with $ d(v_1)=n$ and $d(v_2)=\lceil\frac{k}{2}\rceil-1.$

{\bf Case 2.} $n<\lfloor\frac{3k}{2}\rfloor-1$. Let $B_{n,2}=G[U,V; E]$ be any $(n,2)$-bipartite graph such that $e(B_{n,2})\geq 2(k-1)+1$ with $V=\{v_1, v_2\}$ and $d(v_1)\ge d(v_2)$. Then $d(v_{1})\geq k$ and $d(v_{2})\geq 2(k-1)+1-d(v_2)\ge 2(k-1)+1-n\geq\lceil\frac{k}{2}\rceil$ by $n<\lfloor\frac{3k}{2}\rfloor-1$. Clearly $N(v_{1})\cap N(v_{2})\neq \emptyset$.  Hence $B_{n,2}$  contains all trees in $\mathbb{T}_{k,2}$ as subgraphs. It is easy to see that all extremal graphs are $(n,2)$-bipartite graphs $B_{n,2}=G[U, V; E]$ with $d(v_{1})=d(v_{2})=k-1$, where $V=\{v_{1},v_{2}\}$.

{\bf Case 3.} $n=\lfloor\frac{3k}{2}\rfloor-1$. Then $n+\lceil\frac{k}{2}\rceil-1=2(k-1)$. Let $B_{n,2}=G[U,V; E]$ be any $(n,2)$-bipartite graph with $e(B_{n,2})\geq 2(k-1)+1$. Similarly as Cases 1 and 2, it is easy to see that
$ex(n,2; \mathbb{T}_{k,2})=n+\lceil\frac{k}{2}\rceil-1$ and all extremal graphs are
$(n,2$)-bipartite graphs with  $d(v_1)=n,$ and $d(v_2)=\lceil\frac{k}{2}\rceil-1$, or  $d(v_1)=d(v_2)=k-1$, where $V=\{v_1, v_2\}$.
  \end{proof}

\begin{lemma}
\label{main2,3}
   Let $n\geq m$ and $k\geq m\ge 3$. Then
     \[ex(n,m;\mathbb{T}_{k,2})=\left\{\begin{array}{ll}(m-2)(k-1)+n, & \mbox{for}\ n\geq 2k-1;\\
   (k-1)m, & \mbox{for} \ n<2k-1. \end{array}\right.\]
    Moreover,
     (1). If $n\geq2k-1$, then all  extremal graphs for $\mathbb{T}_{k,2}$ are  $K_{k-1,m-1}\cup K_{n-k+1, 1}$;
           or $B_{n,3}=G[U, V;E]$ with $d(v_1)=d(v_2)=\lceil\frac{k}{2}\rceil-1$ and $d(v_3)=n$, where $k\ge 5$ is odd and $V=\{v_1, v_2, v_3\}$; or $k=3$ and $B_{n,3}=S_{n,3}$.\\
            (2). If $n<2k-1$,
     then the all  extremal graphs for $\mathbb{T}_{k,2}$ are all $(n,m)$-bipartite graphs $B_{n,m}=G[U, V;E]$ with $d(u_{i})\leq k-1$ for $i=1,2,\ldots,n$ and  $d(v_{j})=k-1$ for $j=1,2,\ldots,m,$ where  $U=\{u_1, \ldots, u_n\}$ and $V=\{v_1, \ldots, v_m\}$; or $B_{n,m}=G[U, V;E]$ with $n=2k-2,m=3$, $d(v_1)=d(v_2)=\lceil\frac{k}{2}\rceil-1$, $d(v_3)=2k-2$ and  $k$ being odd, where $V=\{v_1, v_2, v_3\}$.
    \end{lemma}
 \begin{proof}
 (1). Let $n\geq2k-1$. Let $B_{n,m}=G[U,V;E]$ be an $(n,m)$-bipartite graph with $$ e(B_{n,m})=(m-2)(k-1)+n$$ that does not contain all trees in $\mathbb{T}_{k,2}$ as subgraphs.
  Let $U=\{u_{1},\ldots,u_{n}\}$ and $V=\{v_{1},\ldots,v_{m}\}$ with $d(u_{1})\leq\ldots\leq d(u_{n})$ and $d(v_{1})\leq\ldots\leq d(v_{m})$. Since
  $e(B_{n,m})=(m-2)(k-1)+n$ and $n\geq2k-1$, we have $d(v_m)>k-1$. Hence, let $v_{m-l+1},\ldots,v_{m}$ be the vertices in $V$ with degree more than $k-1$.
  Then $N(v_i)\cap N(v_j)=\emptyset$ for $m-l+1\le i\neq j\le m$. Otherwise $B_{n,m}$ contains all trees in $\mathbb{T}_{k,2}$ as subgraphs. So $$\left|\cup_{i=m-l+1}^mN(v_i)\right|=\sum_{i=m-l+1}^{m}d(v_{i}).$$
  Let $$s=max\{n-\sum_{i=m-l+1}^{m}d(v_{i}),\lceil\frac{k}{2}\rceil-1\}.$$ Then
  $\lceil\frac{k}{2}\rceil-1\le s\leq k-1$, otherwise, $$e(B_{n,m})\le n-s+(k-1)(m-l)<(m-2)(k-1)+n,$$ which is a contradiction.
  Furthermore, we have $d(v_{m-l})\le s$, otherwise $$d(v_{m-l})+\sum_{i=m-l+1}^{m}d(v_{i})>n,$$ which implies that there exists a vertex $v_p$  with $m-l+1\le p\le m$ such that $|N(v_{m-l})\cap N(v_p)|\ge 1$. So $B_{n,m}$ contains all trees in $\mathbb{T}_{k,2}$ as subgraphs. On the other hand, we have $$n+(m-2)(k-1)=e(B_{n,m})\leq \sum_{i=m-l+1}^{m}d(v_{i})+(m-l)s\le n +(m-l)(k-1),$$ hence $l=1$. Otherwise, we have $l=2$ with $s=k-1$ and $\sum_{i=m-l+1}^{m}d(v_{i})=n$, which implies that $B_{n,m}$ contains all trees in $\mathbb{T}_{k,2}$ as subgraphs, a contradiction. In addition, if $d(v_{m})<n-(k-1)$, then $e(B_{n,m})\le d(v_m)+(m-1)s< n+(m-2)(k-1)$, which is a contradiction. Furthermore, we claim $d(v_m)=n-k+1$ for $m\ge 4$; and $d(v_m)=n-k+1$ or $d(v_m)=n$ for $m=3$. In fact, if
  $d(v_{m})>n-(k-1)$, then $d(v_{m-1})\le k-2$. Otherwise $v_{m-1}$ and $v_{m}$ has a common neighbor which implies that $B_{n,m}$ contains all trees in $\mathbb{T}_{k,2}$ as subgraphs.
   Moreover, for $m\geq 4,$ we have $d(v_{m-1})\ge \lceil\frac{k}{2}\rceil$. Otherwise $$e(B_{n,m})\le d(v_m)+(m-1)d(v_{m-1})\le n+(m-1)(\lceil\frac{k}{2}\rceil-1)<n+(m-2)(k-1),$$ which is a contradiction. Then $v_{m-1}$ and $v_m$ do not have a common neighbor, otherwise $B_{n,m}$ contains all trees in $\mathbb{T}_{k,2}$ as subgraphs. So $$e(B_{n,m})\le d(v_m)+d(v_{m-1})+(m-2)(k-2)<n+(m-2)(k-1),$$ which is a contradiction. Hence $d(v_{m})=n-k+1$.
    On the other hand, $$n+(m-2)(k-1)\le d(v_m)+(m-1)(k-1)=n+(m-2)(k-1)$$ implies $d(v_1)=\ldots=d(v_{m-1})=k-1$. Moreover, $v_i$ and $v_m$ do not have a common neighbor for $i=1, \ldots, m-1.$ So $$B_{n,m}= K_{k-1,m-1}\cup K_{n-k+1, 1}.$$
If $m=3$ and $d(v_2)\ge \lceil\frac{k}{2}\rceil,$
  then $v_2$ and $v_3$ have no common neighbour. By $$n+(k-1)=e(B_{n,m})=d(v_1)+d(v_2)+d(v_3)\le d(v_1)+n\le n+(k-1),$$ we have $d(v_1)=d(v_2)=k-1, d(v_3)=n-k+1$ (If $d(v_2)\geq k$, then $N(v_1)\cap N(v_2)\neq\emptyset$ or $N(v_1)\cap N(v_3)\neq\emptyset$, contradicting Proposition~\ref{prop}), which implies $$B_{n,m}=K_{n-k+1,1}\cup K_{k-1, m-1}.$$
   If $m=3$ and $d(v_2)\le \lceil\frac{k}{2}\rceil-1$, then by $$n+(k-1)= e(B_{n,m})=d(v_1)+d(v_2)+d(v_3)\le n+2(\lceil\frac{k}{2}\rceil-1),$$ we have $d(v_1)=d(v_2)=\lceil\frac{k}{2}\rceil-1$, $d(v_3)=n$, and $k$ is odd.
    The assertion holds by a simple observation (recall the definition of ``strongly containing'', there are two kinds of possible embeddings considering the partition sets of $T_{k,l}$ and $B_{n,m}$ when $k\leq m$).

(2). Let $n<2k-1$. Let  $B_{n,m}=G[U, V;E]$ be an $(n,m)$-bipartite graph $B_{n,m}=G[U, V;E]$ with  $e(B_{n,m})=(k-1)m$  which does not contain all trees in $\mathbb{T}_{k,2}$ as subgraphs, where  $U=\{u_1, \ldots, u_n\}$ and $V=\{v_1, \ldots, v_m\}$ with $d(v_1)\le \ldots\le d(v_m)$.  We consider the following two cases.

 {\bf Case 1.}  $d(v_m)\ge k.$  Let $$s^{\prime}= max\{n-d(v_m),\lceil\frac{k}{2}\rceil-1\}.$$ Then $d(v_{m-1})\leq s^{\prime}$. Otherwise  $d(v_{m-1})\geq  s^{\prime}+1$, which implies that $d(v_{m-1})+d(v_m)\ge n+1$. So $v_{m-1}$ and $v_{m}$ have at least one common neighbor and $B_{n,m}$ contains all trees in $\mathbb{T}_{k,2}$ as subgraphs, which is a contradiction. Furthermore, we claim
  $s^{\prime}= \lceil\frac{k}{2}\rceil-1$. In fact, if $s^{\prime}=n-d(v_m)\geq\lceil\frac{k}{2}\rceil$, we have
 $$ \begin{array}{lll}
  e(B_{n,m})&\leq &d(v_m)+(m-1)s^{\prime}=d(v_m)+(m-1)(n-d(v_m))\\&\le& k+(m-1)(n-k)
\le k+(m-1)(2k-2-k)<(k-1)m
\end{array}$$
 by $n\le 2k-2$ and $k\geq m\geq 3$, which is a contradiction.
 Hence
 $$(k-1)m=e(B_{n,m})\leq d(v_m)+(m-1)s^{\prime}\le n+ (m-1)(
\lceil\frac{k}{2}\rceil-1)\le 2k-2+(m-1)(\lceil\frac{k}{2}\rceil-1),$$
which implies that
$$m(k-\lceil\frac{k}{2}\rceil)\le 2k-\lceil\frac{k}{2}\rceil-1.$$
Therefore $m=3$, otherwise $m\ge 4$, i.e., $4 (k-\lceil\frac{k}{2}\rceil)\le 2k-\lceil\frac{k}{2}\rceil-1$ which implies $2k+1\le 3\lceil\frac{k}{2}\rceil$, a contradiction. Furthermore, by $m=3$, we have $k+1\le 2\lceil\frac{k}{2}\rceil$ which implies that $k$ is odd.
In addition, by $3(k-1)\le n+2( \lceil\frac{k}{2}\rceil-1)$, we have $n\ge 2k-2$. So $n=2k-2$. Moreover, it is to see that
 $d(v_1)=d(v_2)=\lceil\frac{k}{2}\rceil-1$ and $ d(v_3)=n$.
  If $ k\geq5$, then we are done. If $k=3$, the assertion hold by a simple observation (recall the definition of ``strongly containing'', there are two kinds of possible embeddings considering the partition sets of $T_{k,l}$ and $B_{n,m}$ when $k\leq m$).

{\bf Case 2.}  $d(v_m)\le k-1.$ Then  $(k-1)m=e(B_{n,m})\le (k-1)m$, which implies $d(v_1)=\ldots=d(v_m)=k-1$. By $n\le 2k-2$ and $m\ge 3$, $B_{n,m}$ contains all trees in $\mathbb{T}_{k,2}$ as subgraphs except the tree with one vertex of degree $k$. Hence $d(u_i)\le k-1$ for $i=1, \ldots, n$.
   \end{proof}

\begin{lemma}
\label{n,m}
  (1) Let $n>m\ge 3.$  If $ 3\le k\le m+1$, then  any $(n,m)$-bipartite graph $B_{n,m}=G[U,V;E]$  with $e(B_{n,m})\geq (k-1)n$ contains all trees in $\mathbb{T}_{k,2}$ as subgraphs.\\
   (2) Let  $n=m\geq k-1$, $m\ge 3$ and $k\ge 3$. Then  $
   ex(n,n; \mathbb{T}_{k,2})=(k-1)n$ and all extremal graphs for $\mathbb{T}_{k,2}$ are $(k-1)$-regular bipartite graphs.
 \end{lemma}

 \begin{proof} (1).
  We prove the assertion holds by induction on $m$. If $m=k-1$, then the degree of each vertex in $V$ is  $n$. Hence $B_{n,m}$ contains all trees in $\mathbb{T}_{k,2}$ as subgraphs.   Assume that the assertion holds for less than $m$.
  Let $U=\{u_1, \ldots, u_n\}$ with $d(u_1)\le \ldots\le d(u_p)<k-1$, $d(u_{p+1})=\ldots=d(u_{p+q})=k-1$ and $k-1<d(u_{p+q+1})\le \ldots\le d(u_{p+q+r})$, where $p, q, r\ge 0$ and $p+q+r=n$.
  If $r=0$, then $p=0$ and $q=n$ by $e(B_{n,m})\geq(k-1)n$. Moreover, there is a vertex in $V$ with degree at least $k$ by $n>m$, otherwise $$e(B_{n,m})\le m(k-1)<n(k-1).$$ Hence $B_{n,m}$ contains all trees in $\mathbb{T}_{k,2}$ as subgraphs. So we may assume that $r>0$  and $d(u_{i})=k-1+x_{i} $ for $i=p+q+1, \ldots, n$ where $x_{i}\ge 1$. Furthermore, if there are two vertices of $\{u_{p+q+1},\ldots,u_{n}\}$ which share the same neighbor or one vertex of $\{u_{p+1}, \ldots, u_{p+q}\}$ and one vertex of $\{u_{p+q+1},\ldots,u_{n}\}$  which share the same neighbor, then $B_{n,m}$  contains all trees in $\mathbb{T}_{k,2}$ and the assertion holds. So assume that $$N_{B_{n,m}}(u_{i})\cap N_{B_{n,m}}(u_{j})=\emptyset$$ for $i,j=p+q+1,\ldots,n, i\neq j$, or $i=p+1, \ldots, p+q$ and $j=p+q+1, \ldots, n$. On the other hand, $$(k-1)n=e(B_{n,m})=\sum_{i=1}^nd(u_i)
  \le p(k-2)+q(k-1)+r(k-1)+s=(k-1)n+s-p,$$ where $s=\sum_{i=p+q+1}^nx_i$. So $p\le s$.
Let $$B_{n-p, m-s}=B_{n,m}-\{u_{1}, \ldots, u_{p},v_{p+q+1,1},\ldots,v_{p+q+1,x_{p+q+1}},\ldots,v_{n,1},\ldots,v_{n,x_{n}}\}$$ with $$U_{n-p}=U-\{u_{1},\ldots,u_{p}\},V_{n-s}=V-\{v_{p+q+1,1},\ldots,v_{p+q+1,x_{p+q+1}},\ldots,v_{n,1},\ldots,v_{n,x_{n}}\},$$ where $\{v_{i,1}, \ldots, v_{i, x_i}\}\subseteq N(u_i)$ for $i=p+q+1, \ldots, n$.
Hence $e(B_{n-p,m-s})=(k-1)(n-p)$ and all the vertices of $U_{n-p}$ have degree $k-1$ in $B_{n-p,m-s}$. Note that $$(k-1)(n-p)= e(B_{n-p,m-s}) \le (m-s)(n-p)$$ which implies that $m-s\ge k-1$. Therefore by the induction hypothesis, $B_{n-p,m-s}$ contains all trees in  $\mathbb{T}_{k,2}$ as subgraphs.

(2). It is sufficient to prove that any non-regular bipartite graph $B_{n,n}$  with $e(B_{n,n})\geq(k-1)n$ contains all trees in $\mathbb{T}_{k,2}$ as subgraphs. If $n=m=k-1$, it is trivial. If $n=m\geq k$, then there exists a vertex with degree at least $k$. Suppose that $B_{n,n}$ does not contain all trees in $\mathbb{T}_{k,2}$ as subgraphs. Let $U=\{u_1, \ldots, u_n\}$ with $d(u_1)\le \ldots\le d(u_p)<k-1$, $d(u_{p+1})=\ldots=d(u_{p+q})=k-1$ and $k-1<d(u_{p+q+1})\le \ldots\le d(u_{p+q+r})$, where $p\ge0, q\ge 0, r\ge 1$ and $p+q+r=n$. Recall that the vertices with degree more than $\lceil\frac{k}{2}\rceil-1$ can not share a common neighbor with the vertices with degree more than $k-1$, and hence we can consider the following three cases which are based on the number of neighbors of $\{u_{p+q+1},\ldots,u_{p+q+r}\}$.

{\bf Case 1.} $n-\sum_{i=p+q+1}^{n}d(u_{i})\geq k-1$.  There are at most $n-\sum_{i=p+q+1}^{n}d(u_{i})$ vertices in $U$ with degree $k-1$. Otherwise, the induced subgraph of $B_{n,n}$ with vertex sets $\{u_{p+1},\ldots,u_{p+q}\}$ and $V\setminus \cup_{i=p+q+1}^n N(u_i)$ satisfies (1). Hence $B_{n,n}$ will contain all trees in $\mathbb{T}_{k,2}$ as subgraphs. Therefore $$\begin{array}{lll}e(B_{n,n})&\leq & \sum_{i=p+q+1}^{n}d(u_{i})+(k-1)(n-\sum^{n}_{i=p+q+1}d(u_{i}))
+(k-2)(\sum^{n}_{i=p+q+1}d(u_{i})-r)\\&=&(k-1)(n-r)+r<(k-1)n.\end{array}$$

{\bf Case 2.} $\lceil\frac{k}{2}\rceil\leq n-\sum^{n}_{i=p+q+1}d(u_{i})\leq k-2$. Since $N(u_i)\cap N(u_j)=\emptyset$ for $p+q+1\le i<j\le p+q+r$, we have
\begin{eqnarray*}e(B_{n,n})&\leq&\sum^{n}_{i=p+q+1}d(u_{i})+(n-r)(n-\sum^{n}_{i=p+q+1}d(u_{i}))\\&=&n+(n-r-1)(n-\sum^{n}_{i=p+q+1}d(u_{i}))\\&\leq& n+(n-2)(k-2)\\
&<&(k-1)n.
\end{eqnarray*}

 {\bf Case 3.} $n-\sum^{n}_{i=p+q+1}d(u_{i})\leq\lceil\frac{k}{2}\rceil-1$. We have $$e(B_{n,n})\leq n+(n-r)(\lceil\frac{k}{2}\rceil-1)\leq n+(n-1)(k-2)<(k-1)n.$$
 All of the above three cases contradict $e(B_{n,n})\geq (k-1)n$.
 \end{proof}

\begin{lemma}
\label{main2,4}
  Let $n\geq m\geq k+1\ge 4.$ Then

   $$ex(n,m;\mathbb{T}_{k,2})=\left\{\begin{array}{ll}
   (k-2)(m-1)+n, &  \mbox{for}\ \  n-m\geq k-1;\\
   (k-1)m, & \mbox{for}\ \  n-m<k-1. \end{array}\right.
   $$
   Furthermore, (1). If $n-m\geq k-1$, all extremal graphs for $\mathbb{T}_{k,2}$ are $B_{n,m}=B_{m-1,m-1}^{k-1}\cup K_{1, n-m+1}$, where $B_{m-1,m-1}^{k-1}$ is a $(k-1)$-regular  bipartite graph or $B_{n,m}=S_{n,m}$ when $k=3$.\\
(2).  If $n-m<k-1$, all extremal graphs for $\mathbb{T}_{k,2}$ are $B_{n,m} $ with  $d(u_{i})\leq k-1$ and $d(v_{j})=k-1$ for $i=1,2,\ldots,n;j=1,2,\ldots,m$; or $B_{m+1,m}=S_{m+1,m}$ when $k=3$.
 \end{lemma}
  \begin{proof} {\bf (1).} Suppose that $n-m\geq k-1$.
  Let $B_{n,m}=G[U,V;E]$ be an $(n,m)$-bipartite graph with $(k-2)(m-1)+n$ edges which does not contain all trees in $\mathbb{T}_{k,2}$ as subgraphs.  Then there exists at least one vertex in $V$ with degree at least $k$. Otherwise $$e(B_{n,m})\le (k-1)m= (k-2)(m-1)+n-[(n-m)-(k-1)]-1< (k-2)(m-1)+n,$$ which is a contradiction.
Let $U=\{u_{1},\ldots,u_{n}\}$ and $V=\{v_{1},\ldots,v_{m}\}$
with $d(u_{1})\leq\ldots\leq d(u_{n})$ and $d(v_{1})\leq\ldots\le d(v_{m-l})<k\le d(v_{m-l+1})\le \ldots\leq d(v_{m})$. Moreover,
let $$s=max\{n-\sum_{i=m-l+1}^{m}d(v_{i}),\lceil\frac{k}{2}\rceil-1\},$$ $$U^{\prime}=U\setminus \cup_{i=m-l+1}^m N(v_i) \mbox{ and }
 V^{\prime}=V\setminus \{v_{m-l+1},\ldots,v_{m}\}.$$
Furthermore, we have the following claim.

{\bf Claim:} $m-l\ge k$.

 In fact, suppose that $m-l\le k-1$. Then $l\ge m-k+1\ge 2$ by $m\ge k+1$.
 If $s\ge k-1$, then
 \begin{eqnarray*}e(B_{n,m})&\leq &\sum_{i=m-l+1}^{m}d(v_{i})+(m-l)(k-1)\leq n-(k-1)+(m-l)(k-1)\\&=&(m-1)(k-1)+n-l(k-1)\leq(m-1)(k-1)+n-l(m-2)\\&<&(k-2)(m-1)+n,\end{eqnarray*}
  which is a contradiction. If  $s\le k-2$,  then $d(v_{m-l})\le s$, otherwise $v_{m-l}$ and one vertex in $\{v_{m-l+1},\ldots,v_{m}\}$ have at least one common neighbor, which implies that $B_{n,m}$ contains all trees in $\mathbb{T}_{k,2}$ as subgraphs. Recall that $N(v_i)\cap N(v_j)=\emptyset$ for any $m-l+1\le i<j\le m$. Hence
  \begin{eqnarray*}e(B_{n,m})&\leq &\sum_{i=m-l+1}^{m}d(v_{i})+(m-l)s\leq n+(m-l)s\\&\leq& n+(k-2)(m-2)<(k-2)(m-1)+n,\end{eqnarray*}
   which is a contradiction.  Hence the claim  holds.

  Now we consider the following four cases.

  {\bf Case 1.1.} $s> m-l$. Then $s> m-l\ge k\ge \lceil\frac{k}{2}\rceil-1$ which implies $s=n-\sum_{i=m-l+1}^{m}d(v_{i})$. Hence  \begin{eqnarray*}e(B_{n,m})&\leq&\sum_{i=m-l+1}^{m}d(v_{i})+(m-l)(k-1)<
  n-(m-l)+(m-l)(k-1)\\&=&(k-2)(m-1)+n,\end{eqnarray*}
   which is a contradiction.

  {\bf Case 1.2.}  $s=m-l$. Then $s=m-l\ge k\ge \lceil\frac{k}{2}\rceil-1$ which implies $s=n-\sum_{i=m-l+1}^{m}d(v_{i})$. Hence
      $$(k-2)(m-1)+n=e(B_{n,m})\leq\sum_{i=m-l+1}^{m}d(v_{i})+(m-l)(k-1)=n+(m-l)(k-2).$$ So $l=1$ and $s=m-l=m-1$. Then $d(v_m)=n-(m-1)$ by $s=\max\{n-d(v_m), \lceil\frac{k}{2}\rceil-1\}$. By $$\begin{array}{lll}(k-2)(m-1)+n=e(B_{n,m})&\le& d(v_{m-1})(m-1)+n-(m-1)\\&\le & (k-1) (m-1)+n-(m-1),\end{array}$$ we have $d(v_1)=\ldots=d(v_{m-1})=k-1$.
       Hence $B_{m-1,m-1}$ with $U_{m-1}=U^{\prime}$ and $V_{m-1}=V^{\prime}=\{v_1, \ldots, v_{m-1}\}$ has $(k-1)(m-1)$ edges by Proposition \ref{prop}. Therefore by Lemma~\ref{n,m}, $B_{m-1,m-1}$ is $(k-1)$-regular bipartite graph. So $$B_{n,m}=B_{m-1,m-1}^{k-1}\cup K_{1,n-m+1},$$ where $B_{m-1,m-1}^{k-1}$ is a $(k-1)$-regular bipartite graph.

      {\bf Case 1.3.}  $k-1\leq s<m-l$.  Let $d(v_{1})\leq\ldots\le d(v_p)<k-1=d(v_{p+1})=\ldots=d(v_{m-l})<k\le d(v_{m-l+1})\le \ldots\leq d(v_{m})$, where $p=0$ means that the degree of all vertices in $V$ is at least $k-1$. By
      \begin{eqnarray*}(k-2)(m-1)+n&=&e(B_{n,m})\le p(k-2)+(m-l-p)(k-1)+\sum_{i=m-l+1}^md(v_i)\\&\leq &p(k-2)+(m-l-p)(k-1)+n-s,\end{eqnarray*} we have $m-l-p-s\ge (k-2)(l-1)\ge 0$. Then $s\le m-l-p$.
      On the other hand, since $v_i$ and one vertex of $\{v_{m-l+1},\ldots,v_m\}$ have no common neighbor for $i=p+1, \ldots, m-l$,
      $N(v_i)\subseteq U\setminus \cup_{j=m-l+1}^nN(v_j)$. Hence
      the  $(s, m-l-p)$-bipartite graph $B_{s, m-l-p}=G[U_s, V_{m-l-p};E_{s,m-l-p}]$ with $
      U_s= U\setminus \cup_{j=m-l+1}^nN(v_j)$ and $V_{m-l-p}=\{v_{p+1}, \ldots, v_{m-l}\}$ has $(k-1)(m-l-p)$ edges and $s\ge k-1$. By   Lemma~\ref{n,m}, we have $p=m-l-s>0$ and $B_{s, m-l-p}$ is a $(k-1)$-regular bipartite graph, which contains all trees in $\mathbb{T}_{k,2}$ as subgraphs except the tree with one vertex with degree $k$. Moreover, we have $d(v_1)=\ldots=d(v_p)=k-2$, otherwise $$e(B_{n,m})<(k-2)p+s(k-1)+n-s=n+(k-2)(m-l)\leq n+(k-2)(m-1).$$ Since $k\geq 3$ and $p>0$, then either the neighbors of $v_1$ lie in $\cup_{j=m-l+1}^nN(v_j)$ or $U\setminus \cup_{j=m-l+1}^nN(v_j)$, hence $B_{n,m}$ must contain the tree with one vertex with degree $k$ as a subgraph (recall that $B_{s, m-l-p}$ is a $(k-1)$-regular bipartite graph). Hence, $B_{n,m}$ contains all trees in $\mathbb{T}_{k,2}$ as subgraphs, which is a contradiction.

{\bf Case 1.4.} $s\leq k-2$.  We claim that any vertex in $V^{\prime}$ with degree at most $s$. Otherwise, there must be a vertex in $V^{\prime}$ with degree more than $\lceil\frac{k}{2}\rceil-1 $ sharing at least one common neighbour of a vertex with degree more than $k-1$,  which contradicts Proposition~\ref{prop}. Hence
$$e(B_{n,m})\leq\sum_{i=m-l+1}^{m}d(v_{i})+(m-l)s\leq(k-2)(m-1)+n,$$ with equality holds if and only if $l=1$, $s=k-2$ and $d(v_m)=n$, $d(v_1)=\ldots=d(v_{m-1})=k-2$. If $k\geq4$, then $k-2\geq\lceil\frac{k}{2}\rceil$. Hence $B_{n,m}$ contains all trees in $\mathbb{T}_{k,2}$ as subgraphs. If $k=3$, it is easy to see that  $d(v_m)=n, d(v_1)=\ldots=d(v_{m-1})=1$ (recall the definition of ``strongly containing'', there are two kinds of possible embeddings considering the partition sets of $T_{k,l}$ and $B_{n,m}$ when $k\leq m$, and there are only two graphs in $\mathbb{T}_{3,2}$). Hence $B_{n,m}=S_{n,m}$.

{\bf (2).} Suppose that $n-m\leq k-2$.
  Let $B_{n,m}=G[U,V;E]$ be an $(n,m)$-bipartite graph with $(k-1)m$ edges which does not contain all trees in $\mathbb{T}_{k,2}$ as subgraphs, where $U=\{u_{1},\ldots,u_{n}\}$ and $V=\{v_{1},\ldots,v_{m}\}$. We consider the following two cases.

  {\bf Case 2.1.} The degree of every vertex in $V$ is at most $k-1$.
  Then by $(k-1)m=e(B_{n,m})$, the degree of every vertex in $V$ is $k-1$. Hence by $(k-1)m> m+(k-2)\ge n$, there exist two vertices in $V$ such that they have a common neighbor. So $B_{n,m}$ contains all trees in $\mathbb{T}_{k,2}$ as subgraphs
except the tree with one vertex with degree $k$. Hence $d(u_i)\le k-1$ for $i=1, \ldots, n$, otherwise $B_{n,m}$ contains all trees in $\mathbb{T}_{k,2}$ as subgraphs.

   {\bf Case 2.2.} There exists at least one vertex in $V$ with degree at least $k$. Let
$d(v_{1})\leq\ldots\le d(v_{m-l})<k\le d(v_{m-l+1})\le \ldots\leq d(v_{m})$ with $l\ge 1,$ and $$s=\max\{n-\sum_{i=m-l+1}^md(v_i), \lceil\frac{k}{2}\rceil-1\}.$$
 Then $$|\cup_{i=m-l+1}^mN(v_i)|=\sum_{i=m-l+1}^md(v_i),$$
  $$U^{\prime}=U\setminus \cup_{i=m-l+1}^m N(v_i) \mbox{ and }
 V^{\prime}=V\setminus \{v_{m-l+1},\ldots,v_{m}\}.$$Moreover, let $x_i=d(v_i)-(k-1)$ for $i=m-l+1,\ldots, m$ and $d(v_1)\le \ldots \le d(v_p)<k-1=d(v_{p+1})=\ldots=d(v_{m-l})$, where $p\ge 0.$

 {\bf Subcase 2.2.1.} $n-\sum_{i=m-l+1}^{m}d(v_{i})\geq k-1$. By $$\begin{array}{lll}(k-1)m=e(B_{n,m})&\le & \sum_{i=m-l+1}^{m}d(v_{i})+p(k-2)+(m-l-p)(k-1)\\&=&(k-1)m-p+ \sum_{i=m-l+1}^{m}x_i,\end{array}$$ we have $p\le \sum_{i=m-l+1}^{m}x_{i}$.
  Let  $$B_{n-\sum_{i=m-l+1}^{m}d(v_{i}), m-l-p}=G[U_{n-\sum_{i=m-l+1}^{m}d(v_{i})}, V_{m-l-p};E_{n-\sum_{i=m-l+1}^{m}d(v_{i}), m-l-p}]$$ be a bipartite graph with  $U_{n-\sum_{i=m-l+1}^{m}d(v_{i})}=U^{\prime}$ and $V_{m-l-p}=\{v_{p+1},\ldots,v_{m-l}\}$.
  Then by $p\le \sum_{i=m-l+1}^{m}x_{i}$ and $n-m\le k-2$, we have
  $$m-l-p\ge n-(k-2)-l-\sum_{i=m-l+1}^{m}x_{i}\ge n-(k-1)l-\sum_{i=m-l+1}^{m}x_{i}
  =n-\sum_{i=m-l+1}^{m}d(v_{i}).$$ Clearly the degree of every vertex in $V_{m-l-p}$ is $ k-1$ and $n-\sum_{i=m-l+1}^{m}d(v_{i})\ge k-1$.  Since $B_{n-\sum_{i=m-l+1}^{m}d(v_{i}), m-l-p}$ does not contain all trees in $\mathbb{T}_{k,2}$ as subgraphs, and
  $$e(B_{n-\sum_{i=m-l+1}^{m}d(v_{i}), m-l-p})=(k-1)(m-l-p),$$ by Lemma~\ref{n,m}
  we  have $$m-l-p=n-\sum_{i=m-l+1}^{m}d(v_{i}),  \mbox{ i.e., } \sum_{i=m-l+1}^{m}d(v_{i})=n-m+l+p$$ and $B_{n-\sum_{i=m-l+1}^{m}d(v_{i}), m-l-p}$ is a $(k-1)$-regular bipartite graph, which contains all trees in $\mathbb{T}_{k,2}$ as subgraphs except the tree with one vertex with degree $k$. Hence by $n-m\le k-2$, we have
  \begin{eqnarray*}(k-1)m=e(B_{n,m})&\leq&(k-2)p+(m-l-p)(k-1)+\sum_{i=m-l+1}^{m}d(v_{i})
  \\&=&(k-2)p+(m-l-p)(k-1)+n-m+l+p\\&\le& (k-1)m-(k-2)(l-1),\end{eqnarray*} which implies $l=1$ and $d(v_1)=\ldots d(v_{p})=k-2$. We claim that $p=0$, otherwise, similarly as case 1.3, $v_1$ and one vertex in $\{v_{p+1},\ldots,v_{m}\}$ have a common neighbor, which implies $B_{n,m}$ contains the tree with one vertex with degree $k$ as a subgraph, and hence contains all trees in $\mathbb{T}_{k,2}$ as subgraphs. Moreover, since $p=0,l=1$ and $B_{n-\sum_{i=m-l+1}^{m}d(v_{i}), m-l-p}=B_{m-1,m-1}$ is a $(k-1)$-regular bipartite graph, we get that $d(v_m)=n-m+1\le k-2+1$, which is a contradiction.

 {\bf Subcase 2.2.2.} $n-\sum_{i=m-l+1}^{m}d(v_{i})<k-1$. We claim that $d(v_{i})\leq s $ for $i=1,2,\ldots,m-l.$ Otherwise, there are one vertex in $V^{\prime}$ with degree more than $\lceil\frac{k}{2}\rceil-1$ and one vertex in $\{v_{m-l+1},\ldots, v_m\}$ have a common neighbor. Hence $B_{n,m}$ contains all trees in $\mathbb{T}_{k,2}$ as subgraphs.
  Furthermore, we have $s=\lceil\frac{k}{2}\rceil-1.$ Otherwise, $$s=n-\sum_{i=m-l+1}^{m}d(v_{i})<k-1 \mbox{ and } n-m\le k-2,$$ which implies
  \begin{eqnarray*}(k-1)m=e(B_{n,m})&\le& (m-l)s+\sum_{i=m-l+1}^{m}d(v_{i})=(m-l)s+n-s\\&\le &(m-l-1)s+m+k-2\\&\le &(m-l-1)(k-2)+m+k-2 \\&=&(k-1)m-l(k-2)\\&<&(k-1)m,\end{eqnarray*} a contradiction.  By $s=\lceil\frac{k}{2}\rceil-1, $
 we have  \begin{eqnarray*}(k-1)m&\le &(m-l)s+\sum_{i=m-l+1}^{m}d(v_{i})\\&\le &(m-l)(\lceil\frac{k}{2}\rceil-1)+n\\&\le & (m-l)(k-2)+m+k-2\\&=&(k-1)m-(l-1)(k-2).\end{eqnarray*}
  Then  $l=1$. Furthermore, by $(k-1)m\le  (m-1)(\lceil\frac{k}{2}\rceil-1)+m+k-2$, we have $(k-\lceil\frac{k}{2}\rceil-1)(m-1)\le 0$. Hence $k=3$. Moreover,  by $2m=(k-1)m\le (m-1)(\lceil\frac{3}{2}\rceil-1)+ n\le (k-1)m=2m$, we have $n=m+1$. Therefore $d(v_m)=n$ by $2m\le (m-1)+d(v_m)$ and $d(v_m)\le n=m+1$. Then it is easy to see that $d(v_{m-1})=\ldots=d(v_1)=1$. Since there are only two graphs in $\mathbb{T}_{3,2}$, by an easy observation, we have $d(u_{m+1})=m,d(u_{m})=\ldots=d(u_{1})=1$. The assertion holds.
 \end{proof}

 Now we are ready to present the proof of Theorem~\ref{main2}.

 \begin{proof}
Theorem~\ref{main2} follows from Lemmas~\ref{main2,1},~\ref{main2,2},~\ref{main2,3} and \ref{main2,4}.\end{proof}

\section{Proof of Theorem~\ref{main3}}
Since there are exactly three trees $G_{1}$,$G_{2}$ and $G_{3}$  (see Figure 1) in $\mathbb{T}_{3,3}$, we have the following result for  small $m$ and $n$.
\begin{lemma}\label{small}
Let $n\ge m$ and $1\leq m \leq 4$. Then
$$ex(n, m;\mathbb{T}_{3,3})=\left\{\begin{array}{ll} 9,& \mbox{for}\ \ m=n=4;\\
2n, &\mbox{for}\ \ 2\le m\le 4, (n,m)\neq (4,4);\\
n, &\mbox{for}\ \ m=1.
\end{array}\right.$$
Furthermore, (1). If $m\le 2$, all the extremal graphs for  $\mathbb{T}_{3,3}$ are $K_{n,m}$.

(2). If $m=3$ and $n=3$, then all the extremal graphs for $\mathbb{T}_{3,3}$ are $B_{3,3}=G[U,V;E]$ such that any vertex in $U$ or $V$ has degree two.

(3). If $m=3$ and $n\ge 4$, then all the extremal graphs for $\mathbb{T}_{3,3}$ are $B_{n,3}=G[U,V;E]$ such that any vertex in $U$ has degree two.

(4). If $m=4$ and $n=4$, then the extremal graph for $\mathbb{T}_{3,3}$ is $G_1^{\prime} $ (see Figure 2).

(5). If $m=4$ and $n=5$, then  all the extremal graphs for $\mathbb{T}_{3,3}$ are $B_{5,4}=G[U,V;E]$ such that any vertex in $U$ has degree two and $G_2^{\prime}$. (see Figure 2).

(6). If $m=4$ and $n\ge 6$, then all the extremal graphs for $\mathbb{T}_{3,3}$ are
$B_{n,4}=G[U,V;E]$ such that any vertex in $U$ has degree two.

\begin{picture}(200,70)(0,-40)
\put(20,24){\line(1,0){16}}
\put(20,24){\line(1,-1){16}}
\put(20,24){\line(1,-2){16}}
\put(20,8){\line(1,1){16}}
\put(20,8){\line(1,-2){16}}
\put(20,-8){\line(1,1){16}}
\put(20,-8){\line(1,-1){16}}
\put(20,-24){\line(1,1){16}}
\put(20,-24){\line(1,0){16}}

\put(20,8){\circle*{2}}
\put(36,8){\circle*{2}}
\put(20,-8){\circle*{2}}
\put(36,-8){\circle*{2}}
\put(20,-24){\circle*{2}}
\put(36,-24){\circle*{2}}
\put(20,24){\circle*{2}}
\put(36,24){\circle*{2}}
\put(28,-32){$G^{\prime}_{1}$}

\put(100,24){\line(1,0){12}}
\put(100,24){\line(1,-1){12}}
\put(100,24){\line(1,-2){12}}
\put(100,24){\line(1,-3){12}}
\put(100,12){\line(1,1){12}}
\put(100,12){\line(1,-3){12}}
\put(100,0){\line(1,1){12}}
\put(100,0){\line(1,-2){12}}
\put(100,-12){\line(1,1){12}}
\put(100,-12){\line(1,-1){12}}
\put(100,-24){\line(1,0){12}}
\put(100,-24){\line(1,1){12}}

\put(100,24){\circle*{2}}
\put(112,24){\circle*{2}}
\put(100,12){\circle*{2}}
\put(112,12){\circle*{2}}
\put(100,0){\circle*{2}}
\put(112,0){\circle*{2}}
\put(100,-12){\circle*{2}}
\put(112,-12){\circle*{2}}
\put(100,-24){\circle*{2}}
\put(112,-24){\circle*{2}}
\put(106,-32){$G^{\prime}_{3}$}

\put(60,24){\line(1,0){12}}
\put(60,24){\line(1,-1){12}}
\put(60,24){\line(1,-2){12}}
\put(60,12){\line(1,1){12}}
\put(60,12){\line(1,-2){12}}
\put(60,0){\line(1,1){12}}
\put(60,0){\line(1,-1){12}}
\put(60,-12){\line(1,1){12}}
\put(60,-12){\line(1,0){12}}
\put(60,-24){\line(1,1){12}}

\put(60,24){\circle*{2}}
\put(72,24){\circle*{2}}
\put(60,12){\circle*{2}}
\put(72,12){\circle*{2}}
\put(60,0){\circle*{2}}
\put(72,0){\circle*{2}}
\put(60,-12){\circle*{2}}
\put(72,-12){\circle*{2}}
\put(60,-24){\circle*{2}}
\put(66,-32){$G^{\prime}_{2}$}

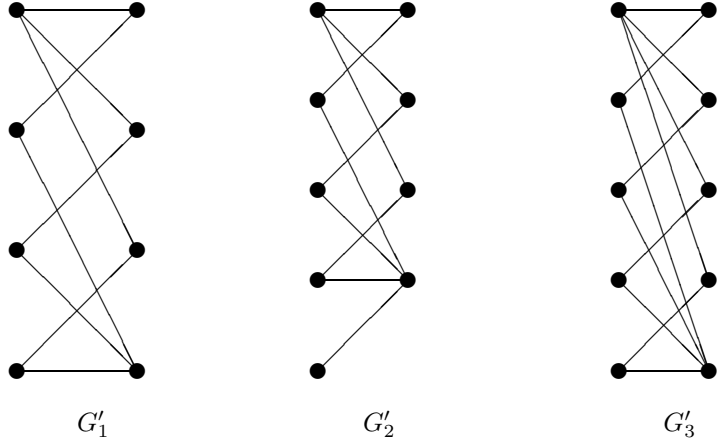
\put(40,-40){Figure 2: Some extremal graphs}
\end{picture}
\end{lemma}
\begin{proof} It is trivial for $m\le 2$. For $m=3$ and $n=3$, it follows from Theorem~\ref{main1}. For $m=3$ and $n\ge 4$, let $B_{n,3}=G[U,V;E]$ be any bipartite graph with $e(B_{n,3})\ge 2n$ which does not contain all trees in $\mathbb{T}_{3,3}$ as subgraphs. If there exists a vertex in $U$ with degree three, then it is easy to check that $B_{n,3}$ contains $G_{1},G_{2}$ in $\mathbb{T}_{3,3}$ as subgraphs. For $G_{3}=P_{6}$, by theorem~\ref{P2l}, $B_{n,3}$ contains $P_6$ as a subgraph. Hence, $B_{n,3}$ contains all trees in $\mathbb{T}_{3,3}$ as subgraphs, a contradiction. Since $B_{n,3}=G[U,V;E]$ in which degree of each vertex in $U$ is two does not contain $G_1$ as a subgraph, we conclude that the assertion holds for $m=3$ and $n\ge 4$. For $m=4$, and $n=4, 5$, let $B_{4,4}=G[U,V;E]$ be any bipartite graph with $e(B_{4,4})\ge9$ and $B_{5,4}$ be any bipartite graph with $e(B_{5,4})\ge10$ which does not contain all trees in $\mathbb{T}_{3,3}$ as subgraphs. If there are one vertex $u_{i}$ in $U$ and one vertex $v_{j}$ in $V$ with $d(u_{i})\ge 3,d(v_{j})\ge 3$ and $u_i$ hits $v_j$, then it is easy to see that $B_{4,4}$ and $B_{5,4}$ both contain $G_1$ and $G_2$ as subgraphs. Furthermore, by Theorem~\ref{P2l}, $B_{4,4}$ and $B_{5,4}$ contain $P_6$ as a subgraph. Hence $B_{4,4}$ and $B_{5,4}$ contains all trees in $\mathbb{T}_{3,3}$ as subgraphs. It is easy to see that the extremal graphs in (4) and (5) do not contain $G_1$ as a subgraph and they are the only possible extremal graphs ($G^\prime_1$, $G^\prime_2$ and $B_{5,4}=G[U,V;E]$ such that any vertex in $U$ has degree two.), hence the assertion holds. For $m=4$ and $n\ge 6$, let $B_{n,4}=G[U,V;E]$ be any bipartite graph with $e(B_{n,4})\ge2n$ which does not contain all trees in $\mathbb{T}_{3,3}$ as subgraphs. If there exists a vertex in $U$ with degree at least three, then it is easy to check that $B_{n,4}$ contains $G_{1},G_{2}$ in $\mathbb{T}_{3,3}$ as subgraphs. Moreover, by Theorem~\ref{P2l}, $B_{n,4}$ contains $P_6$ as a subgraph. So the assertion holds. \end{proof}

\begin{lemma}\label{4.1}
Let $n\ge m\ge 5$. Then  $$ex(n,m;G_{1})=2n+2m-8.$$
Furthermore, if $n=m=5$, all  extremal graphs for $G_1$ are $K_{2,3}\cup K_{2,3}$ and $G^{\prime}_{3}$.
If $m\ge 5, n\ge 6$, then  all  extremal graphs  for $G_1$ are $K_{2,n-2}\cup K_{2,m-2}$.
\end{lemma}

\begin{proof}
 If $n=m=5$, it is easy to see that the assertion holds. Now assume that $m\ge 5$ and $n\ge 6$.  Let $B_{n,m}=G[U,V;E]$ be an $(n,m)$-bipartite graph with $e(B_{n,m})=2n+2m-8$ which does not contain $G_1$ as a subgraph. Let $U_{1}$ and $V_1$ be set of the vertices in $U$ and $V$ with degree more than two, respectively. Denote by $U_{2}=U\setminus U_1$ and $V_2=V\setminus V_1$. Since $B_{n,m}$ does not contain $G_1$ as a subgraph, any vertex in $U_1$ does not hit any vertex in $V_1$. So $$2n+2m-8=e(B_{n,m})\leq 2(n-|U_{1}|)+2(m-|V_{1}|)-e(U_{2},V_{2}).$$ Hence $$2(|U_1|+|V_1|)+e(U_{2},V_{2})\le 8.$$ On the other hand, we have $|U_1|\ge 1$, otherwise $e(B_{n,m})\le 2n$ which contradicts $e(B_{n,m})=2n+2m-8$ and $m\ge5$. Furthermore, we have $|V_1|\ge 2$, otherwise there is at most one vertex in $V$ with degree at most $n-1$ and $e(B_{n,m})\le n-1+2(m-1)$ which contradicts $e(B_{n,m})=2n+2m-8$.  Therefore $|U_1|=|V_1|=2$, otherwise by $|U_1|+|V_1|\le 4$, we have $|U_1|=1$ and $2\le |V_1|\le 3$ which implies $$e(B_{n,m})\le m-2+2(n-1)<2n+2m-8,$$ a contradiction.
 Moreover, $e(U_2, V_2)=0$. So each vertex  in $U_1$  ($V_1$) hits each vertex in $V_2$ ($U_2$), respectively. Hence $B_{n,m}=K_{2,n-2}\cup K_{2,m-2}$.
 \end{proof}

\begin{lemma}\label{4.2} Let $n\ge m\ge5$. If an $(n,m)$-bipartite graph $B_{n,m}=G[U,V;E]$ with $e(B_{n,m})\geq2n+2m-8$ is not an extremal graph in Lemma~\ref{4.1}, then $B_{n,m}$ contains $G_2$ as a subgraph.
\end{lemma}
\begin{proof}
 Suppose that $B_{n,m}=G[U,V;E]$ with $e(B_{n,m})\geq2n+2m-8$ is not an extremal graph in Lemma~\ref{4.1} and does not contain $G_2$ as a subgraph.
 By Lemma~\ref{4.1}, $B_{n,m}$ contains $G_1$ as a subgraph. Hence there exists a vertex  $u$ in $U$ with degree at least three and  a vertex  $v$ in $V$ with degree at least three such that $u$ hits $v$. Since $B_{n,m}$ does not contain $G_2$ as a subgraph, all vertices in $N(u) \cup N(v)\setminus\{u,v\}$ are pendent vertices. Hence the induced subgraph by the vertex set $N(u)\cup N(v)$ is a component of $B_{n,m}$ with $d(u)+d(v)$ vertices and $d(u)+d(v)-1$ edges. Let $C_1,\ldots,C_p$ be the components of $B_{n,m}$ such that there are $u_i\in U\cap C_i$ and $v_i\in V\cap C_i$ with $u_i$ hits $v_i$ and $d(u_i)\geq 3$, $d(v_i)\geq 3$. Let $B_{n_1,m_1}=G[U_{n_1},V_{m_1}; E_{n_1,m_1}]$ be the union of all these  $p$ components with $|U_{n_1}|=n_1$ and $|V_{m_1}|=m_1.$ Then
 $e(B_{n_1, m_1})=n_1+m_1-p$. Furthermore any two vertices with degree at least three in $B_{n,m}-U_{n_1}\cup V_{m_1}=B_{n_2, m_2}$ are not adjacent, where $n_2=n-n_1$ and $m_2=m-m_1$. Hence $B_{n,m}-U_{n_1}\cup V_{m_1}$ does not contain $G_1$ as a subgraph. Without loss of generality, let $n_2\ge m_2$.  If $m_2 \ge 5$, then $$e(B_{n,m}-U_{n_1}\cup V_{m_1})\ge e(B_{n,m})-e(B_{n_1,m_1})
 \ge2n+2m-8-(n_1+m_1-p)>2n_2+2m_2-8.$$ If $m_2=4,n_2\ge 5$, then $$e(B_{n,m}-U_{n_1}\cup V_{m_1})\ge e(B_{n,m})-e(B_{n_1,m_1})
 \ge2n+2m-8-(n_1+m_1-p)>2n_2.$$ If $m_2=4,n_2=4$, then $$e(B_{n,m}-U_{n_1}\cup V_{m_1})\ge e(B_{n,m})-e(B_{n_1,m_1})
 \ge2n+2m-8-(n_1+m_1-p)>9.$$ If $m_2=3$, then $$e(B_{n,m}-U_{n_1}\cup V_{m_1})\ge e(B_{n,m})-e(B_{n_1,m_1})
 \ge2n+2m-8-(n_1+m_1-p)>2n_2.$$ Then by Lemmas~\ref{small} and \ref{4.1}, $B_{n_2, m_2}$ contains $G_1$ as a subgraph, which is a contradiction. If $m_2\le2$, then $$e(B_{n,m}-U_{n_1}\cup V_{m_1})\ge e(B_{n,m})-e(B_{n_1,m_1})
 \ge 2n+2m-8-(n_1+m_1-p)> m_2n_2,$$ a contradiction. So the assertion holds.
\end{proof}
 Now we are ready to prove Theorem~\ref{main3}.

 \begin{proof}
 If $3\leq m\le 4$, then the assertion holds from Lemma~\ref{small}. If $m\ge5$, by Lemmas~\ref{4.1}, \ref{4.2} and Theorem~\ref{P2l}, any bipartite graph $B_{n,m}$ with $e(B_{n,m})\geq2(n+m-4)$ edges contains $G_1$, $G_2$ and $G_3$ as subgraphs except the extremal graph in Lemma~\ref{4.1}. So we finish the proof. Moreover, the extremal graphs are described in Lemmas~\ref{small} and \ref{4.1}.
 \end{proof}

  {\bf Remark} Let $B^{k-1}_{m-l+1,m-l+1}$ be a $(k-1)$-regular bipartite graph with order $2m-2l+2$ and $B_{n,m;k-1}=G[U,V;E]$ be a bipartite graph with $d(u_{i})\leq k-1$, $d(v_{j})=k-1$ for $i=1,2,\ldots,n;j=1,2,\ldots,m,$ where $U=\{u_1,\ldots,u_n\}$, $ V=\{v_1,\ldots,v_j\}$. It is easy to see that
   $K_{n-l+1,l-1}\cup K_{l-1,m-l+1}$,  $K_{n-m+l-1,l-1}\cup B^{k-1}_{m-l+1,m-l+1}$
and  $B_{n,m;k-1}$  does not contain all trees in $\mathbb{T}_{k,l}$ as subgraphs.  Base on the proof of Theorems~\ref{main2} and \ref{main3}, we propose the following conjecture:
\begin{conjecture}
Let $n\geq m,k \geq l, n\geq k$ and $m\geq l.$

(1).
If $k<2l-2,m\geq 2l$, then
\[ex(n,m,\mathbb{T}_{k,l})=(l-1)(n+m-2l+2).\]

(2). If $k\geq2l-1,m-l+1\geq k-1,n-m+l-1\geq k$, then
\[ex(n,m,\mathbb{T}_{k,l})=(k-1)(m-l+1)+(l-1)(n-m+l-1).\]

(3). If $k\geq2l-1,m-l+1\geq k-1,n-m+l-1\leq k-1$, then
\[ex(n,m,\mathbb{T}_{k,l})=(k-1)m.\]
\end{conjecture}

\noindent{\bf Acknowledgements:}

  The authors would like to thank the anonymous referee for many helpful and constructive suggestions to an earlier version of this paper, in particular for giving new short proofs of Lemmas~\ref{strongly,m,m} and \ref{main1'}.

\end{document}